\documentclass{amsart}
\usepackage{amssymb}
\usepackage{amsfonts}
\usepackage{amsmath}
\usepackage{mathrsfs}

\newcommand{\myself}{\author{Gianluca Cassese}
                     \address{Universit\`{a} Milano Bicocca and University of Lugano}
                     \email{gianluca.cassese@unimib.it}
                     \curraddr{Department of Statistics, Building U7, Room 2081, via Bicocca 
                               degli Arcimboldi 8, 20126 Milano - Italy}}

\newtheorem{theorem}{Theorem}
\theoremstyle{plain}

\newtheorem{corollary}{Corollary}

\newtheorem{lemma}{Lemma}

\newtheorem{remark}{Remark}

\numberwithin{equation}{section}
\newcommand{\LIM}{\mathrm{LIM}} 
 
\newcommand{\B}{\mathfrak{B}} 
 
\newcommand{\F}{\mathscr{F}}
\newcommand{\R}{\mathbb{R}}

\newcommand{\restr}[2]{#1\vert #2} 
\newcommand{\lrestr}[2]{\left.#1\right\vert #2}

\newcommand{\abs}[1]{\vert #1\vert} 
\newcommand{\dabs}[1]{\left\vert #1\right\vert} 
\newcommand{\norm}[1]{\Vert #1\Vert}

\newcommand{\condP}[2]{P(#1\vert #2)} 
 
\newcommand{\lcondP}[2]{P\left(\left.#1\right\vert #2\right)}

\newcommand{\net}[3]{\left\langle #1_{#2}\right\rangle_{#2\in #3} } 
\newcommand{\nnet}[3]{\left\langle #1\right\rangle_{#2\in #3} } 
\newcommand{\seq}[2]{\net{#1}{#2}{\mathbb{N}}} 
\newcommand{\sseq}[2]{\nnet{#1}{#2}{\mathbb{N}}} 
\newcommand{\seqn}[1]{\seq{#1}{n}} 
\newcommand{\sseqn}[1]{\sseq{#1}{n}}

\newcommand{\Der}[2]{d #1\left/d #2\right.}
\newcommand{\dDer}[2]{\dfrac{d #1}{d #2}}

\newcommand{\set}[1]{\mathbf{1}_{#1}}
\newcommand{\E}{\mathscr E}
\newcommand{\M}{\mathscr M}
\newcommand{\D}{\mathscr D}
\newcommand{\normQ}[1]{\norm{#1}_{\mathcal Q}}
\newcommand{\Pd}{\mathscr P^d}
\newcommand{\Fd}[1]{\F_{\delta_{#1}}}
\newcommand{\X}[1]{X_{\delta_{#1}}}
\newcommand{\V}[1]{\dabs{\X{#1}-\condP{\X{#1+1}}{\Fd{#1}}}}

\begin{document}

\title{Quasimartingales with a Linearly Ordered Index Set}
\myself
\date
\today
\subjclass[2000]{Primary 28A12, 60G07, 60G20.} 

\keywords{Doob Meyer decomposition, Natural increasing process, Potential,
Quasi-potential, Rao decomposition, Riesz decomposition.}

\maketitle

\begin{abstract}
We prove a version of Rao decomposition for quasi-martingales indexed by a linearly ordered set.
\end{abstract}

\section{Introduction.}

Rao's Theorem \cite[Theorem 2.3, p. 89]{rao quasi} asserts that, under the 
\textit{usual conditions,} each quasi-martingale decomposes into the difference 
of two positive supermartingales. This decomposition turns out to be a crucial step
in the proof of the Bichteler Dellacherie theorem, the fundamental theorem of 
semimartingale theory. We offer here a simple proof of this classical result 
without assuming right continuity neither for the filtration nor for the processes. 
Moreover, we allow the index set to be just a linearly ordered set. The setting 
is in fact the same as that proposed in \cite{STAPRO}, where a version of the 
Doob Meyer decomposition was obtained. A suitable example of a linearly ordered
index set would be the whole real line.

In our general model the main source of difficulty (and of interest) lies in the 
need to forsake the stopping machinery, a most useful tool in stochastic analysis. 
We rather exploit the measure theoretic approach to stochastic processes, inaugurated by 
Dol\'{e}ans-Dade \cite{doleans} and followed by many others, including Metivier and 
Pellaumail \cite{metivier pellaumail} and Dellacherie and Meyer \cite{dellacherie meyer}.
In fact, we argue, the measure representation of processes is useful also in the absence
of countable additivity, on which the literature has focused hitherto. The main result
of this paper, Theorem \ref{theorem ba}, establishes that quasimartingales are 
isometrically isomorphic to locally countably additive measures. We show that all 
classical decompositions follow straightforwardly.

\section{The Model.}
The following notation will be convenient. $(\Omega,\F,P)$ will be a standard probability 
space, $\Delta$ a linearly ordered set and $(\Fd{}:\delta\in\Delta)$ an increasing family 
of sub $\sigma$ algebras of $\F$. All $\sigma$ algebras considered include the corresponding 
family of $P$ null sets. $L^p(\Fd{})$ will be preferred to $L^p(\Omega,\Fd{},P)$, $\B(\Fd{})$
to $\B(\Omega,\Fd{})$ (the space of bounded, $\Fd{}$ measurable functions on $\Omega$) and
$\bar\Omega$ to $\Omega\times\Delta$. $\bar\F$ will be the augmentation of 
$\F\otimes 2^\Delta$ with respect to sets with $P$ null projection on $\Omega$.
We shall also write $]\delta_1,\delta_2]$ for the (possibly empty) set 
$\{\delta\in\Delta:\delta_1<\delta\leq\delta_2\}$. 

All stochastic processes $(\X{}:\delta\in\Delta)$ to be mentioned will be adapted, 
i.e. such that $X_\delta$ is $\Fd{}$ measurable for each $\delta\in\Delta$, but no
form of right continuity is assumed. Two processes $X$ and $Y$ are considered as 
equal up to modification whenever $P(X_\delta=Y_\delta)=1$ for all $\delta\in\Delta$ 
and all decompositions introduced later should be understood to be unique in the above
sense.

Of special importance is the collection $\D$ of finite subsets of $\Delta$, with
generic element $d=\{\delta_1\leq\ldots\leq\delta_{N+1}\}$, a directed with respect
to inclusion. $\D_\delta$ (resp. $\D_\delta^{\delta'}$) will
denote the family of those sets $\{\delta_1\leq\ldots\leq\delta_{N+1}\}\in\D$ such that 
$\delta_1=\delta$ (resp. $\delta_1=\delta$ and $\delta_{N+1}=\delta'$).

To each $d\in\D$ we associate the collection
\begin{equation}
\label{P(d)}
\Pd=\left\{\bigcup_{n=1}^NF_n\times]\delta_n,\delta_{n+1}]
:F_n\in\Fd{n},\ n=1,\ldots,N\right\}  
\end{equation}
and define $\mathscr P=\bigcup_{d\in\D}\Pd$ and $\E=\bigcup_{d\in\D}\B(\Pd)$. Abusing notation, 
we use the symbol $\Pd$ also for the operator $\Pd:\B(\bar\F)\to\B(\Pd)$ defined implicitly as
\begin{equation}
\label{Proj}
\Pd(U)=\sum_{n=1}^N\condP{U_{\delta_{n+1}}}{\Fd{n}}\set{]\delta_n,\delta_{n+1}]}
\end{equation}
A process $A$ is increasing if $P(0\leq A_{\delta_1}\leq A_{\delta_2}) =1$ for 
$\delta_1\leq\delta_2$ and $\inf_{\delta\in\Delta}P(A_\delta)=0$; it is integrable 
if $\sup_{\delta\in\Delta}P(\abs{A_\delta})<\infty$. An increasing process $A$ 
is natural if
\begin{equation}
\label{natural}
P\left(b\int fdA\right)=\underset{d\in\D}\LIM\ P\int\Pd(b)fdA\qquad b\in L^\infty(\F),
\ f\in\E
\end{equation}
where $\LIM$ denotes here the Banach limit operator. 

\section{Quasi-martingales.}
For each $d\in\D$ the $d$-variation of a process $X$ is defined to be
\begin{equation}
\label{variation}
V^d(X)=\sum_{n=1}^N\V{n}\quad\text{where}\quad d=\{\delta_1,\ldots,\delta_{N+1}\} 
\end{equation} 
A process $X$ is a quasi-martingale if 
\begin{equation}
\label{quasi}
\normQ{X}=\sup\left\{P\left(V^d(X)\right):d\in\D\right\}<\infty
\end{equation}
Quasi-martingales have been introduced and studied by Fisk \cite{fisk}, Orey \cite{orey}
and Rao \cite{rao quasi} who proved the classical decomposition theorem in its general form. 
Related results were obtained by Stricker \cite{stricker characterisation} and \cite{stricker quasi}
who proved uniqueness of the Rao decomposition. Our definition follows Rao \cite{rao quasi} 
and differs from the one adopted by Stricker \cite[p. 55]{stricker quasi} and by Dellacherie 
and Meyer \cite[p. 98]{dellacherie meyer}, which requires quasi-martingales to be integrable, 
a property that we do not impose here. However, in our setting supermartingales need not be 
quasi-martingales if not integrable. 

Quasi martingales form a normed space with respect to $\normQ{\cdot}$ if we only identify 
processes which differ by a martingale. We will denote such space as $\mathcal Q$. A quasi-%
potential $X$ is a quasi-martingale which admits a sequence $\seqn{\delta}$ in $\Delta$ such 
that $\lim_nP(\abs{X_{\bar\delta_n}})=0$ for any sequence $\seqn{\bar\delta}$ such that
$\bar\delta_n\geq\delta_n$ for $n=1,2,\ldots$. A potential is at the same time a 
quasi-potential and a positive supermartingale. Of course the difference of
two potentials is a quasi-potential.

We shall make use of the following inequality, where $d'=\{\delta_1,\ldots,\delta_{N+1}\}$
and $d=\{\delta_1,\delta_{N+1}\}$
\begin{eqnarray}
\label{domination}
\notag\lcondP{V^{d'}(X)}{\Fd{1}}&=&\lcondP{\sum_{n=1}^N\V{n}}{\Fd{1}}\\
\notag&\geq& \sum_{n=1}^N\dabs{\condP{\X{n}}{\Fd{1}}-\condP{\X{n+1}}{\Fd{1}}}\\
&\geq& \dabs{\sum_{n=1}^N\lcondP{\X{n}-\X{n+1}}{\Fd{1}}}\\
\notag &=&\dabs{\X{1}-\lcondP{\X{N+1}}{\Fd{1}}}\\
\notag &=&V^d(X)
\end{eqnarray}
We draw from (\ref{domination}) the following implications
\begin{lemma}
\label{lemma preliminary}
Let $X$ be a quasi-martingale and $\seqn{\delta}$ an increasing sequence. 
The net $\nnet{ \condP{V^d(X)}{\Fd{}}}{d}{\D_\delta}$ and the sequence 
$\sseqn{\condP{X_{\delta_n\vee\delta}}{\Fd{}}}$ both converge in $L^1(\Fd{})$.

\end{lemma}
\begin{proof}
Let $d'=\{\delta_{n_1},\ldots,\delta_{n_{K+1}}\},\ d=\{\delta_1,\ldots,\delta_{N+1}\}\in\D_\delta$ and 
$d\leq d'$. Then 
$$
V^{d'}(X)=\sum_{k=1}^K\V{k}=\sum_{n=1}^N\sum_{\{\delta_n\leq\delta_{n_k}\leq\delta_{n+1}\}}
\dabs{\X{n_k}-\condP{\X{n_{k+1}}}{\Fd{n_k}}}
$$
so that, by (\ref{domination}),
\begin{eqnarray*}
\lcondP{V^{d'}(X)}{\Fd{}}&=&\sum_{n=1}^N\lcondP{\lcondP{\sum_{\{\delta_n\leq\delta_{n_k}\leq\delta_{n+1}\}}
\dabs{\X{n_k}-\condP{\X{n_{k+1}}}{\Fd{n_k}}}}{\Fd{n}}}{\Fd{}}\\
&\geq&\sum_{n=1}^N\lcondP{\V{n}}{\Fd{}}\\
&=&\lcondP{V^d(X)}{\Fd{}}
\end{eqnarray*}
The net $\nnet{ \condP{V^d(X)}{\Fd{}}}{d}{\D_\delta}$ is thus increasing but $P(V^d(X))\leq\normQ{X}$. 
Convergence in $L^1$ follows then from \cite[Lemma 1]{STAPRO}. Again by (\ref{domination}), we get that 
$
P\left(\sum_{n=1}^N\dabs{ \condP{\X{n}}{\Fd{}}-\condP{\X{n+1}}{\Fd{}} }\right)
\leq P(V^d(X))\leq\normQ{X}
$
which proves the second claim.
\end{proof}

\section{A Characterisation.}
Quasi-martingales can be characterised as elements of the space $ba(\bar\F)$, i.e. as
bounded finitely additive measures on $\bar\F$. $x\in ba(\bar\F)$ is locally countably 
additive, in symbols $x\in\mathscr M^{loc}$, if $x$ is countably additive in restriction to 
$\Pd$ for each $d\in\D$. 
\begin{theorem}
\label{theorem ba}
There is an isometric isomorphism between $\mathcal Q$ and $\mathscr M^{loc}$
determined by the identity
\begin{equation}
\label{isometry}
x(f)=-\underset{d\in\D}\LIM \ P\int\Pd(f)dX\qquad f\in\B(\bar\F)
\end{equation}
In particular,
\item[(\textit{i})] each $x\in \M^{loc}$ corresponds via (\ref{isometry}) to
one and only one integrable quasi-potential $X$;
\item[(\textit{ii})] $x\in \M^{loc}_+$ if and only if it corresponds via (\ref{isometry})
to an integrable potential.
\end{theorem}

\begin{remark}
\label{remark variation}
If $x$ and $X$ are as in (\ref{isometry}) then necessarily
$x(f)=\LIM_{d\in\D}\ x(\Pd(f))$. Thus, if $\abs{x}$ denotes the total variation
measure of $x$ and if $H\in\mathscr P$, 
$$
\abs{x}(H)=\sup_{\{g\in\B(\bar\F),\abs{g}\leq1\}} x(\set{H}g)
=\sup_{\{g\in\B(\bar\F),\abs{g}\leq1\}}\LIM_{d\in\D}\ x(\Pd(g)\set{H})
\leq\sup_{\{h\in\E,\abs{h}\leq1\}}x(h\set{H})
$$
In other words, the restriction of $\abs{x}$ to $\mathscr P$ coincides with total 
variation of $\restr{x}{\mathscr P}$.
\end{remark}

\begin{proof}
Let $X$ be a quasi-martingale. In view of the inequality
$\dabs{P\int\Pd(f)dX}\leq\norm{f}_{\B(\bar\F)}\normQ{X}$, 
the right-hand side of (\ref{isometry}) defines a continuous linear 
functional on $\B(\bar\F)$ and thus admits the representation as an integral
with respect to some $x\in ba(\bar\F)$, \cite[corollary IV.5.3, p. 259]{bible}
with $\norm{x}\leq\normQ{X}$. The correspondence between $X$ and $x$ is linear 
by the properties of the Banach limit. Given that $P\int hdX\geq\normQ{X}-\epsilon$ 
for all $\epsilon>0$ and some $h\in\E$ with $\norm{h}\leq1$, 
then $\norm{x}=\normQ{X}$. Fix $d\in\D$ and let $\seq{H}{k}$ be a decreasing sequence 
in $\Pd$ with empty intersection. With no loss of generality we may assume that 
$d=\{\delta,\delta'\}$. By definition, $H_k=F_k\times]\delta,\delta']$ for some 
$F_k\in\F_{\delta}$. Then, by Remark \ref{remark variation}
\begin{eqnarray*}
\abs{x}(H_k)&=&\sup_{\{h\in\E,\abs{h}\leq1\}}x(h\set{H_k})\\
&\leq&\sup_{d\in\D_\delta^{\delta'}}P
\left(F_k\sum_{n=1}^N\abs{X_{\delta_n}-\condP{X_{\delta_{n+1}}}{\F_{\delta_n}}}\right)\\
&=&\sup_{d\in\D_\delta^{\delta'}}P\left(F_k\condP{V^d(X)}{\Fd{}}\right)
\end{eqnarray*}
By Lemma \ref{lemma preliminary}(\textit{i}), $\condP{V^d(X)}{\Fd{}}$ converges in 
$L^1(\Fd{})$ and therefore $P(F_kV^d(X))$ converges with $d\in\D_\delta^{\delta'}$ 
uniformly in $k$ so that $\lim_k\abs{x}(H_k)\leq
\lim_k\lim_{d\in\D_\delta^{\delta'}}P(F_kV^d(X))=
\lim_{d\in\D_\delta^{\delta'}}\lim_kP(F_kV^d(X))=0$
which proves that $\abs{x}\in \M^{loc}$ and, \textit{a fortiori}, that
the same is true of $x$. If $X$ is an integrable potential, it is then a quasi-martingale
associated to some $x\in\mathscr M^{loc}_+$.

Fix now $x\in \M^{loc}$ and let $\sseqn{\delta_x(n)}$ and $\sseqn{\delta^x(n)}$
be monotonic sequences in $\Delta$ such that 
\begin{equation}
\label{deltan}
\lim_n\abs{x}(]\delta_x(n),\delta^x(n)])=\sup_{\delta,\delta'\in\Delta}\abs{x}(]\delta,\delta'])
\end{equation}
Define $x^\delta,\ x_\delta\in ba(\F)$ implicitly as
\begin{equation}
\label{xd}
x^\delta(F)=\lim_nx(F\times]\delta_x(n),\delta]),\quad 
x_\delta(F)=\lim_nx(F\times]\delta,\delta^x(n)])\qquad\qquad F\in\F 
\end{equation}
Given that $x$ is locally countably additive, then $\restr{x_\delta}{\Fd{}}\ll\restr{P}{\Fd{}}$
and we denote by $X_\delta$ the corresponding Radon Nykodim derivative. Clearly, 
$P(\abs{X_\delta})\leq\norm{x}$. If $F\in\Fd{} $ and $\delta'\geq\delta$, then 
$P(\set{F}(X_\delta -\condP{X_{\delta'}}{\Fd{} })=x(F\times]\delta,\delta'])$ 
so that (\ref{isometry}) is satisfied. Moreover 
$\lim_nP(\abs{X_{\bar\delta_n}})\leq\lim_n\lim_k\abs{x}(]\delta^x(n),\delta^x(k)])=0$
for any sequence $\seqn{\bar\delta}$ such that $\bar\delta_n\geq\delta^x(n)$ for all $n$. If 
$x\geq0$ then it is obvious that $X$ is a positive supermartingale. If $Y$ were another
integrable quasi-potential corresponding to $x$ via (\ref{isometry}) then $X-Y$ would be at the same
time a quasi-potential and a martingale. But then, for all $\delta\in\Delta$ and some sequence 
$\seqn{\bar\delta}$, 
$P(\abs{X_\delta-Y_\delta})\leq\lim_nP(\abs{X_{\bar\delta_n\vee\delta}-Y_{\bar\delta_n\vee\delta}})=0$.
\end{proof}
A conclusion implicit in Theorem \ref{theorem ba} is that all quasi-potentials
are integrable. Another consequence is the following version of a result of Stricker 
\cite[Theorem 1.2, p. 55]{stricker quasi}.
\begin{corollary}[Stricker]
\label{corollary stricker}
Let $X$ be a quasi-martingale, $(\mathscr G_\delta:\delta\in\Delta)$ a sub filtration of 
$(\F_\delta:\delta\in\Delta)$ and define $X_\delta^{\mathscr G}=\condP{X_\delta}{\mathscr G_\delta}$. 
Then the process $X^{\mathscr G}$ is itself a quasi-martingale.
\end{corollary}
\begin{proof}
The restriction to the subfiltration preserves local countable additivity of the corresponding measure.
\end{proof}
The issue of the invariance of the process structure with respect to a change of the underlying
filtration was also addressed in \cite{JOTP}.

\section{Decompositions.}
The following version of Riesz decomposition follows immediately from Theorem \ref{theorem ba}.
\begin{corollary}
\label{corollary riesz}
A process $X$ is a quasi-martingale if and only if it decomposes into the sum 
\begin{equation}
\label{riesz}
X=M+B
\end{equation} 
of a martingale $M$ and a quasi-potential $B$. The decomposition is then unique. 
Moreover, if $X$, $M$ and $B$ are as in (\ref{riesz}), then
\begin{enumerate}
\item[(\textit{i})] $X$ is a (positive) supermartingale if and only if $B$ is a potential 
(and $M$ is positive);
\item[(\textit{ii})] If for any $\delta$ there is an integer $n$ such that $\delta^x(n)\geq\delta$, 
then $M$ is uniformly integrable if and only if $\{X_{\delta^x(n)}:n=1,2,\ldots\}$ is so.
\end{enumerate}
\end{corollary}
\begin{proof} 
By Theorem \ref{theorem ba}, $X$ corresponds to some $x\in\M^{loc}$ and $x$, in turn, to a unique 
quasi-potential, $B$. Given that $\normQ{X-B}=0$, then $M=X-B$ is indeed a martingale. To be 
more explicit, we write
\begin{equation}
\label{riesz decomposition}
X_\delta=\lim_n\lcondP{X_{\delta^x(n)\vee\delta}}{\Fd{}}+\lrestr{\dDer{x_\delta}{P}}{_{\Fd{}}} 
=M_\delta+B_\delta
\end{equation}
the sequence $\sseqn{\delta^x(n)}$ being defined as in (\ref{deltan}). To see that
(\ref{riesz decomposition}) is well defined, observe that the limit appearing in it 
exists in $L^1(\Fd{})$ (by Lemma  \ref{lemma preliminary}) and that
$x_\delta(F)=\lim_nx(F\times]\delta,\delta^x(n)])=\lim_nx(F\times]\delta,\delta^x(n)\vee\delta])
=\lim_nP(\set{F}(X_\delta-X_{\delta^x(n)\vee\delta}))$ for all $F\in\Fd{}$. Moreover, 
$$
\dabs{P(\set{F}(M_\delta-M_{\delta'}))}=
\lim_n\dabs{P(\set{F}(X_{\delta^x(n)\vee\delta}-X_{\delta^x(n)\vee\delta'}))}\leq
\dabs{x}((]\delta^x(n),\delta^x(n)\vee\delta']))=0
$$
so that $M$ is a martingale and it is positive if $X\geq0$. 

Claim (\textit{i}) is a consequence of the fact that $x\geq0$ whenever $X$ is a supermartingale. 
If $\{X_{\delta^x(n)}:n=1,2,\ldots\}$ is uniformly integrable then
$P(\abs{M_\delta}\set{F})\leq\sup_nP(\abs{X_{\delta^x(n)}}\set{F})$
for any $F\in\Fd{}$, which implies uniform integrability of $M$. The converse
is obvious given that the collection $\{M_{\delta^x(n)}:n=1,2,\ldots\}$ is
uniformly integrable by assumption while $\lim_nP(\abs{B_{\delta^x(n)}})=0$ by
construction.
\end{proof} 

It is noteworthy that, when the index set is order dense, e.g. $\Delta=\R$, uniform 
integrability of the martingale $M$ does not require more than uniform integrability of 
$X$ \textit{along a given sequence}, i.e. of a countable set. In \cite{STAPRO} it was 
shown that the class $D$ property could likewise be restricted to consider only stopping 
times with countably many values. 

\begin{corollary}[Rao and Stricker]
\label{corollary rao}
A process $X$ is a quasi-potential (resp. an integrable quasi-martingale) if 
and only if it decomposes into the difference 
\begin{equation}
\label{rao}
X=X'-X''
\end{equation}
of two potentials (resp. positive, integrable supermartingales) such that
$\normQ{X}=\normQ{X'}+\normQ{X''}$. Any other decomposition $Y'-Y''$
of $X$ as the difference of two potentials (resp. positive, integrable supermartingales)
is such that $Y'-X'$ and $Y''-X''$ are potentials (resp. positive, integrable, supermartingales). 
\end{corollary}
\begin{proof}
Let $X$ be a quasi-martingale isomorphic to $x$, $x'-x''$ be its Jordan decomposition and $X'$ 
and $X''$ the associated potentials as of Theorem \ref{theorem ba}.(\textit{i}). Then,
$\normQ{X}=\norm{x}=\norm{x'}+\norm{x''}=\norm{X'}+\norm{X''}$.
If $X$ is a quasi potential, then (\ref{rao}) follows from (\ref{riesz decomposition}).
If $X$ is an integrable quasi-martingale with Riesz decomposition $M+B$, then $M$ is integrable.
Define 
$$M'_\delta=\lim_n\lcondP{M^+_{\delta^x(n)\vee\delta}}{\Fd{}}\quad\text{and}\quad 
M''_\delta=\lim_n\lcondP{M^-_{\delta^x(n)\vee\delta}}{\Fd{}}$$
Observe that convergence takes place in $L^1(\Fd{})$ as the corresponding sequences are 
increasing but norm bounded. Let $B=B'-B''$ be the decomposition of the quasi-potential 
$B$ into the difference of two potentials, following from the previous step and set
$X'=M'+B'$ and $X''=M''+B''$. Clearly, $X'$ and $X''$ are positive supermartingales and 
satisfy (\ref{rao}). Moreover $\normQ{X}=\normQ{B'}+\normQ{B''}=\normQ{X'}+\normQ{X''}$.
Let $Y'-Y''$ be a decomposition of $X$ into two potentials (resp. positive, integrable 
supermartingales). Then $Y'$ corresponds to some $y'\in \M^{loc}$ such that necessarily 
$y'\geq x'$ so that the potential part of $Y'$ exceeds the corresponding component of $X'$. 
To conclude that the same is true of the martingale part (if any), observe that $Y'\geq X^+$ 
so that
\begin{eqnarray*}
\lim_n\lcondP{Y'_{\delta^x(n)\vee\delta}}{\Fd{}}&\geq&
\lim_n\lcondP{X^+_{\delta^x(n)\vee\delta}}{\Fd{}}\\
&\geq&\lim_k\lim_n\lcondP{\lcondP{X_{\delta^x(n)\vee\delta}}{\F_{\delta^x(k)\vee\delta}}^+}{\Fd{}}\\
&=&\lim_k\lcondP{M'_{\delta^x(k)\vee\delta}}{\Fd{}}\\
&=&M'_\delta
\end{eqnarray*}
We then obtain from (\ref{riesz decomposition}) that $Y'-X'$ is indeed a potential (resp. 
positive, integrable supermartingale). It is easily seen that the same is true of $Y''$.
\end{proof}

Local countable additivity has not been much studied in the literature, given the 
exclusive attention paid to the class $ca(\mathscr P)$ in the literature. Such
attention was motivated by the proof offered by Dol\'{e}ans-Dade \cite{doleans} that the Doob 
Meyer decomposition is equivalent, \textit{under the usual conditions}, to countable 
additivity over the predictable $\sigma$ algebra. It was Mertens \cite{mertens} the 
first to provide a version of this result without invoking such regularity assumptions, 
a result later proved in greater generality by Dellacherie and Meyer \cite[Theorem 20, p. 414]{dellacherie meyer}. 
A proof, in the setting of linearly ordered index set, was given in \cite{STAPRO} 
based on a suitable extension of the \textit{class D property}. A purely measure theoretic 
proof was provided in \cite[theorem 4, p. 597]{JOTP} for processes indexed by the 
positive reals. This same argument will now be easily adapted to the general setting 
considered here.

\begin{theorem}
\label{theorem doob meyer}
Let $x\in\mathscr M^{loc}_+$ be isomorphic to a potential $X$ and define $x_\F$ implicitly as
\begin{equation}
\label{marginal}
x_\F(F)=x(F\times\Delta)\qquad F\in\F
\end{equation}
Then $x_\F\ll P$ if and only if there exist $M\in L^1_+$ and an 
increasing, integrable, natural process $A$ such that
\begin{equation}
\label{doob meyer}
X_\delta=\condP{M}{\Fd{}}-A_\delta\qquad P\ a.s.,\ \delta\in\Delta
\end{equation}
The decomposition (\ref{doob meyer}) is unique.
\end{theorem}
\begin{proof}
Assume that $x_\F\ll P$ and recall the definition (\ref{xd}) of $x^\delta\in ba(\F)_+$. 
The inequality $x^\delta\leq x_\F$ guarantees that $x^\delta\ll P$ for all $\delta\in\Delta$. 
Let $A_\delta$ be the corresponding Radon Nykodim derivative. Then $P(0\leq A_\delta\leq A_{\delta'})=1$
for all $\delta\leq\delta'$. Choose $b\in L^\infty(\F)$ and $f\in\E$. 
Then $f=f\set{]\delta',\delta]}$ for some $\delta',\ \delta\in\Delta$ and (\ref{isometry}) implies
\begin{eqnarray*}
P\left(b\int fdA\right)&=&x^\delta (bf)\\
&=&x(bf\set{]\delta',\delta]})\\
&=&\underset{d\in\D}\LIM\ x\left(\Pd(b)f\set{]\delta',\delta]}\right)\\
&=&\underset{d\in\D}\LIM\ x^\delta\left(\Pd(b)f\right)\\
&=&\underset{d\in\D}\LIM\ P\int\Pd(b)fdA
\end{eqnarray*}
so that $A$ is natural (and thus adapted). Moreover, letting $M=\Der{x_\F}{P}$ and
$F\in\Fd{}$ we have
$$
P(\set{F}X_\delta)=\lim_nx(F\times]\delta,\delta^x(n)\vee\delta])+\lim_nP(\set{F}X_{\delta^x(n)\vee\delta})
=x_\delta(F)=x_\F(F)-x^\delta(F)=P(\set{F}(M-A_\delta))
$$
Suppose that $N-B$ is another such decomposition. Then $A$ and $B$ are both natural and
then for each $F\in\Fd{}$,
$$
P(\set{F}A_\delta)=\underset{d\in\D}\LIM\lim_n P\int_{\delta_x(n)}^\delta\Pd(F)dA
=\underset{d\in\D}\LIM\lim_nP\int_{\delta_x(n)}^\delta\Pd(F)dB
=P(\set{F}B_\delta)
$$
\end{proof}

\end{document}